\documentclass[12pt, reqno]{amsart}
\usepackage{amsfonts}
\usepackage{bbm}
\usepackage{amscd,amsfonts}
\usepackage{amssymb, eucal, amsfonts, amsmath, xypic, latexsym}
\usepackage{pifont}
\usepackage{color}
\usepackage{mathrsfs}
\usepackage{amsthm,indentfirst,bm,fancyhdr,dsfont}
\usepackage{graphicx}
\usepackage[all]{xy}
\usepackage[CJKbookmarks=true]{hyperref}

\usepackage{mathrsfs,color}

\newtheorem{thm}{Theorem}[section]
\newtheorem{lem}[thm]{Lemma}
\newtheorem{cor}[thm]{Corollary}
\newtheorem{prop}[thm]{Proposition}

\newtheorem{super-derivations}[thm]{Super-Derivation}

 \theoremstyle{definition}
\newtheorem{example}[thm]{Example}

\newtheorem{defn}[thm]{Definition}

\newtheorem{Kac claim}[thm]{Kac claim}

\numberwithin{equation}{thm}

\newcommand{\de}{\delta}
\def\N{\mathscr N}
\def\ggg{\mathfrak{g}}

\def\ppp{\mathfrak{p}}
\def\nnn{\mathfrak{n}}
\def\mmm{\mathfrak{m}}
\def\zzz{\mathfrak{z}}
\def\hhh{\mathfrak{h}}
\def\uuu{\mathfrak{u}}
\def\ker{{\rm Ker\,}}

\def\ad{{\rm ad}}
\def\ch{{\rm ch}}

\def\Hom{{\rm Hom}}

\def\pr{\partial}

\def\bbf{\mathbb F}
\def\bbz{\mathbb Z}
\def\bbc{\mathbb C}

\def\co{\mathcal O}
\def\calr{\mathcal{R}}
\def\calo{\mathcal{O}}

\newcommand{\mf}{\mathfrak}

\newcommand{\al}{\alpha}

\newcommand{\ot}{\otimes}

\newcommand{\dis}{\displaystyle}
\newcommand{\ga}{\gamma}
\newcommand{\mbf}{\mathbf}

\def\N{\mathbb{N}}

\def\bbc{\mathbb{C}}
\def\Z{\mathbb{Z}}
\def\bbz{\mathbb{Z}}

\newcommand{\op}{\oplus}
\newcommand{\cd}{\cdot}
\newcommand{\bgop}{\bigoplus}
\newcommand{\mscr}{\mathscr}
\newcommand{\mcal}{\mathcal}

\newcommand{\ov}{\overline}
\newcommand{\la}{\lambda}

\def\Res{\text{Res}}

\def\sfp{\textsf{P}}
\def\sfn{\textsf{N}}

	\def\lc{\mathcal{L\kern -.3em   C}}

	\def\bz{{\bar{0}}}
	\def\bo{{\bar{1}}}
	\def\hmod{\text{-\bf{mod}}}

	\def\mscrw{\mathscr{W}}
	\def\bfB{\mathbf{B}}

	\def\bfg{\mathbf{g}}
	
	\def\bfn{\mathbf{n}}
	
	\def\Spec{\textsf{Spec}}
	
	\def\hmod{\text{-}\mathbf{mod}}

\begin{document}
		
		\title[Representations of Lie superalgebras]{Whittaker Category and finite $W$-superalgebras for Cartan type Lie superalgebras}
		
		\author{Priyanshu Chakraborty, Yuhui shen and Bin Shu}
		

		
		
		\begin{abstract} Let $W(n)$ be the finite-dimensional simple Lie superalgebra of fundamental type in the Cartan type series of Kac's  classification result \cite{Kac77} over an algebraically closed field of characteristic $0$. Let $\mathbf{g}$ be the graded-zero part of $W(n)$ which is isomorphic to $\mathfrak{gl}(n)$.
			In the first part of this paper,  following the basic idea of taking the ``minimal" parabolic subalgebra $\mathsf{P}$ as a working platform in \cite{DSY} we introduce the Whittaker category $\mscrw$ for representations of $W(n)$ associated with a nilpotent element $e$ in $\mathbf{g}_0$ and with $W(n)_{-1}$. This Whittaker category  turns out to be close to the classical Whittaker category  McDowell and Mili\v{c}i\'{c}-Soergel studied in \cite{Mc} and  \cite{MS}, respectively (or see \cite{Back}). We finally classify the simple objects in $\mscrw$. In the second part, we introduce the finite $W$-algebra associated with $e$,  we then establish a generalized Skryabin's equivalence between the representation category  of the finite $W$-superalgebra and  the category  $\mscrw'$ of so-called weakened Whittaker modules over $W(n)$. Here $\mscrw'$ naturally contains $\mscrw$ as a full subcategory.
		\end{abstract}
		
		\maketitle

		\section*{Introduction and preliminaries}

		\subsection{} In \cite{Kac77}, Kac classified all finite-dimensional simple Lie superalgebras over $\bbc$. His result shows that all finite-dimensional simple Lie superalgebras over $\bbc$ which are not Lie algebras,  are classified into two classes: classical Lie superalgebras and Cartan type Lie superalgebras.  The study on classical Lie superalgebras over $\bbc$  increasingly draws attention.  Lots of results make it in a great progress.   Comparatively, the study on Cartan type Lie algebras are lack of attention.

		\subsubsection{} In \cite{DSY}, Feifei Duan, Yufeng Yao and the last author of the present paper first obtained an interesting classifying result on parabolic subalgebras of graded Cartan type Lie superalgebras, then they were aware of the significance of ``minimal" parabolic subalgebras in their representation categories. They  initiated the study of parabolic Bernstein-Galfeld-Galfeld category $\mathcal{O}$ for graded Cartan type Lie superalgebras over $\bbc$, associated with the minimal parabolic subalgebras. There the authors classified all blocks of the category $\mathcal{O}$ and precisely described them, and finally obtained character formulas  for indecomposable tilting modules and projective modules.

		\subsubsection{} In the present paper,  following the basic idea of taking ``minimal" parabolic subalgebras as a basic platform  \cite{DSY} we introduce our  Whittaker category $\mscrw$ for the fundamental Cartan type Lie superalgebra $W(n)$. Such a Whittaker category  turns out to be close to the classical Whittaker category  McDowell and Mili\v{c}i\'{c}-Soergel studied in \cite{Mc} and  \cite{MS}, respectively (or see \cite{Back}). In the first part of the present paper (\S\ref{sec: sec one}-\S\ref{sec: sec three}), we  then describe the Whittaker category $\mscrw$,  furthermore, we introduce  a generalized Backelin functor connecting the parabolic BGG category $\calo$ and $\mscrw$.
 However, when we tried to introduce the corresponding finite $W$-superalgebra and establish generalized Skryabin's equivalence, we found that the above aim could be only partially gained.  As a result, in the second part (\S\ref{sec: second part}),  giving up to make use of the minimal parabolic subalgebra $\mathsf{P}$,  we instead consider the so-called weakened Whittaker modules only associated with a nilpotent element $e\in \mathbf{g}$ as introduced in the abstract (see Remark \ref{sec: example} for further explanation). The category of the weakened  Whittaker modules is bigger than $\mscrw$, which is denoted by $\mscrw'$. Then we consider the finite $W$-algebra corresponding to $\mscrw'$. The main results are listed below.
		\begin{itemize}
			\item[(1)] Classify all simple objects for $\mscrw$, up to isomorphisms (see Proposition \ref{Thm: unique irr quo}, Theorem \ref{thm: cat decomp} and Theorem \ref{thm: 3.2}).

			\item[(2)]  The objects of this category generally does not have finite-length composition series unless $n=2$.

			\item[(3)]  Establish a generalized Skryabin's equivalence between $\mscrw'$ and the module category of the finite $W$-algebra corresponding to $\mscrw'$ (see Theorem \ref{thm: skryabin eq}).
		\end{itemize}

		\subsection{}		Throughout the paper, we assume that $n$ is a positive integer and $\calr:= \bigwedge_\bbf(y_1,\ldots,y_n)$ is the polynomial superalgebra with indeterminants $  y_s, s=1,\ldots,n$ over an algebraically closed field $\bbf$. The degree conventions  $\deg y_j= 1$ naturally determine a $\bbz$-grading on $\calr$. Hence $\calr$ is endowed with $\bbz$-graded structure,   this $\bbz$-graded structure gives rise to the super space structure of $\calr$. This is to say,
		$$\calr =\sum_{i=0}^{N} \calr_i$$
		with $\calr_0=\bbf$,  $\calr_i=\text{span}_\bbf\{ y_{k_1}\wedge \cdots \wedge y_{k_i}\mid  1\leq k_1<\cdots<k_i\leq n\}$, $i=1,\ldots, N:=2^{n}-1$, and
		$$\calr=\calr_\bz\oplus \calr_\bo$$ for
		$\calr_\bz=\dis\sum_{i=\text{even}}\calr_i$, $\calr_\bo=\dis\sum_{i=\text{odd}}\calr_i$. In the subsequence, the above polynomial $y_{k_1}\wedge \cdots \wedge y_{k_i}$ in $\calr$ is simply written as $y_{k_1}\cdots y_{k_i}$.

		Denote by $W(n)$ the space of super derivations of $\mcal R$ which is required to be a free $\calr$-module of rank $n$ with basis $D_i={\partial\over{\partial y_i}}$ for $i=1,\ldots,n$.  By convention,  the degree of $D_i$ is equal to $1$.
		
		Naturally,  $W(n)$ is endowed with  a $\bbz$-graded structure with gradation
		\begin{align*}
			W(n)_i=\sum_{k=1}^n\calr_{i+1}D_k.
		\end{align*}
		Moreover, $W(n)$ is also a superspace endowed with a $\Z_2$-grading compatible its $\bbz$-grading. This is to say,
		\begin{align*}
			&W(n)_\bz=\sum_{i=\text{even}} W(n)_i, \cr
			&W(n)_\bo=\sum_{i=\text{odd}} W(n)_i.
		\end{align*}
		By convention, $\deg f$ denotes $i$ if  $f$ lies in $\calr_i$.  Then $W(n)$ becomes a Lie superalgebra over $\bbf$ with Lie bracket  by defining via
		$$[fD_s, gD_t]=fD_s(g)D_t-(-1)^{(\deg f-1)(\deg g-1)}gD_t(f)D_s$$
		for any homogeneous polynomials $f\in \calr_{\deg f}$ and $g\in \calr_{\deg g}$.
		
		\subsubsection{}	From now on, we always simply denote $W(n)$ by $\ggg$.   Then $\ggg=\dis \bigoplus_{i\geq -1}\mf g_i$ and $\ggg_0 \simeq \mf{gl}(n)$ as Lie algebras.
		Recall that $\ggg_0$ admits a canonical Cartan subalgebra $\hhh$ which is  spanned  by $y_iD_i$ for
		$i=1,\ldots,n$. With respect to $\hhh$, $\ggg_0$ admits a root system $\Phi=\{\epsilon_i-\epsilon_j\mid 1\leq i\ne j\leq n\}$ and the corresponding positive root system $\Phi^+=\{ \epsilon_i-\epsilon_j\mid 1\leq i<j\leq n\}$. Let $\Delta$ be the set of all simple roots, i.e. $\Delta=\{\epsilon_i-\epsilon_{i+1}\mid i=1,\ldots,n-1\}$ and $\nnn^+=\sum_{i<j}\bbf y_iD_j$, $ \nnn^-=\sum_{i>j}\bbf y_iD_j$ be the sum of all positive and negative root spaces of $\ggg_0$ respectively.  Then $\ggg_0=\hhh+\sum_{\alpha\in \Phi}\ggg_{0,\alpha}$  with  $\ggg_{0,\alpha}=\bbf y_iD_j$ for $\alpha=\epsilon_i-\epsilon_j\in \Phi$, $(i\ne j)$.
		
		\subsubsection{} Let $W$ be the Weyl group of $\ggg_0$ which is isomorphic to the symmetry group $\mathfrak{S}_n$ of degree $n$. Let $\rho$ be half of the sum of positive roots of $\ggg_0$. For $\sigma\in W$ and $\lambda\in \hhh^*$, denote $\sigma\bullet\lambda=\sigma(\lambda+\rho)-\rho$.

		\section{The Whittaker categories for $\ggg$}\label{sec: sec one}
		Keep the above notations and assumptions. Denote by $U(\ggg)\hmod$ the category of $U(\ggg)$-modules.
		\subsection{Whittaker categories for the Cartan type Lie superalgebra $\ggg=W(n)$}		
		
		\begin{defn}\label{def: whittaker}
			A Whittaker category $\mathscr{Wh}$ is a full subcategory of $U(\ggg)\hmod$,  with objects $M$ satisfying the following axioms:
			
			\begin{itemize}
				
				\item[(W1)]  $M$ is finitely generated over $U(\ggg)$;

				\item[(W2)]  $M$ is locally finite over $\nnn^+$ and locally nilpotent over $\ggg_{-1}$, i.e., for any $v\in M$, $U(\nnn^+)v$ is finite-dimensional, and there exists some positive integer $N_X$ such that  $X^{N_X}v=0$ for any $X\in \ggg_{-1}$;
				
				\item[(W3)]  $M$ is locally finite over $Z(\ggg_0):=\text{Cent}(U(\ggg_0))$, i.e., for any $v\in M$, $Z(\ggg_0)v$ is finite-dimensional.
			\end{itemize}
		\end{defn}
		Obviously,  the category $\co$ defined in \cite{DSY} is a subcategory of $\mathscr{Wh}$. Furthermore, the following fact is very fundamental and important.

		\begin{lem}\label{lem: abelian} The category $\mscrw$ is an abelian category.
		\end{lem}
		
		\begin{proof} By a routine check according to the definition, it is easily proved.
		\end{proof}
		
		
		\begin{example}\label{sec: example}
			Let $\eta:\nnn^+\rightarrow \bbf$ be a character of $\nnn^+$, which can be regarded as a linear function on $\nnn^+\slash [\nnn^+,\nnn^+]$ (associated with the nondegenerate trace form $(X,Y):=\mathsf{tr}(XY)$ on $\ggg_0\cong \mathfrak{gl}(n)$,  $\eta$ uniquely corresponds to a nilpotent element $e\in \ggg_0$ such that $\eta=(e,\cdot)$). So $\eta$ determines a subset $I$ of $\Delta$, consisting of $\alpha\in \Delta$ with  $\eta(\ggg_{0,\alpha})\ne0$, here $\ggg_{0,\alpha}$ denotes the root space of $\mf g_0$ corresponding to the root $\al$. Let $\ggg^\eta_0$ be the standard Levi subalgebra of $\ggg_0$ corresponding to the subsystem I of $\Delta$. Then $\ggg^\eta_0$ admits   $\nnn^+_\eta$ the sum of all positive root spaces associated with its positive root system $\Phi(\ggg_0^\eta)^+\subset \Phi^+$.  Denote by $\ppp_0^\eta$ the corresponding parabolic subalgebra of $\ggg_0$, this means,
			$$\ppp_0^\eta=\ggg_0^\eta\oplus\uuu^+$$
			with $\uuu^+$ denoting the nilpotent radical of $\ppp_0^\eta$ which is the sum of all root spaces $\ggg_{0,\gamma}$ corresponding to $\gamma\in\Phi^+\backslash \bbz I$. Consider the center of $U(\ggg_0^\eta)$ which is  denoted by $Z(\ggg_0^\eta)$.
			Let $\xi_\eta^{\#}: Z(\ggg^\eta_0) \to S(\mf h) $ be the Harish-Chandra homomorphism of $U(\ggg^\eta_0)$ satisfying $\xi_\eta^{\#}(z)-z \in U(\ggg^\eta_0)\mf n_\eta^+$. It induces a map from $\hhh^*$ to the spectrum of the maximal ideals of $Z(\ggg_0^\eta)$, which is denoted by $\xi_\eta: \mf h^* \to \Spec(Z(\ggg_0^\eta))$.

			For the 1-dimensional representation $\bbf_{\eta}$ of $ \nnn^+_\eta$, we introduce a $U(\ggg_0^\eta)$-module defined by
			$$ Y(\eta,\la)={U(\ggg_0^\eta) }/{\xi_\eta(\la-\rho_0)U(\ggg^\eta_0)}  \otimes_{U(\mf n_\eta^+)} \bbf_{\eta},$$
			where $\rho_0$ is half sum of positive roots of $\ggg_0\cong\mf{gl}(n)$, $\la\in \hhh^*$. Notice that the first projection
			$\mf p_0^\eta= \ggg_0^\eta\oplus\uuu^+ \twoheadrightarrow \ggg_0^\eta$ defines a $U(\mf p_\eta)$-module structure on $Y(\eta,\la)$.
			It is known from \cite{Mc,MS} that $M^0(\eta,\la)=U(\ggg_0) \otimes_{U(\ppp_0^\eta)} Y(\eta,\la)$  is a Whittaker module for $\ggg_0$, which admits a unique irreducible quotient $L^0(\eta,\la)$.

			Let $\sfp:= \mf g_{-1} \op \mf g_0 $, $\sfn:= \mf g_{-1} \op \mf n^+\subset \sfp$. Then both $\sfp$ and $\sfn$ are Lie sub-superalgebra of $\ggg$, respectively.  By extending $\eta$ to a function on $\sfn$ with letting $\eta(\ggg_{-1})=0$,
			we get a character  $\eta: \sfn \to \bbf$.
			Now we set
			$$ \mcal M(\eta,\la)= U(\mf g) \ot_{U(\sfp)} L^0(\eta,\la) ,$$
			where $L^0(\eta,\la)$ extends to a $\sfp$-module by letting $\mf g_{-1}$ trivially act.
		\end{example}
		
		\subsection{} Now we further investigate $\mcal{M}(\eta,\la)$.  Denote by $\text{ch}(\nnn^+)$ the set of characters on $\nnn^+$.
 Note that associated with  $\eta\in \text{ch}(\nnn^+)$, we have a Levi subalgebra $\ggg_0^\eta$ as in Example \ref{sec: example}, and then the corresponding Weyl group which is denoted by $\mcal W_\eta$.
		\begin{prop}\label{Thm: unique irr quo} The following statements hold.
			
			\begin{itemize}
				\item[(1)] For $\eta\in \text{ch}(\nnn^+)$ and $\la\in \hhh^*$,  $\mcal M(\eta,\la)$ is an object in $\mscr W$.	
				\item[(2)] $\mcal M(\eta,\la)$ has a unique irreducible  quotient, denoted by $\mcal L(\eta,\la)$.
				\item[(3)] The following are equivalent.
				\begin{itemize}
					\item[(a)]				$\mcal M(\eta,\la) \simeq \mcal M(\eta,\mu)$;
					\item[(b)]  $\mcal L(\eta,\la) \simeq \mcal L(\eta,\mu)$;
					\item[(c)]  $\mcal W_\eta \bullet \la = \mcal W_\eta \bullet\mu.$
				\end{itemize}
			\end{itemize}
			
		\end{prop}	
		\begin{proof}
			For the part (1), first note that $\mcal M(\eta,\la)$ is finitely generated $U(\mf g)$-module since it is induced from   $L^0(\eta,\la)$, and the latter is irreducible $\mf g_0$-module.  So $\mcal M(\eta,\la)$ satisfies  the axiom (W1) in Definition \ref{def: whittaker}.   By the definition of $\mcal M(\eta,\la)$,   the trivial action of $\mf g_{-1}$ on $L^0(\eta,\la)$ and the locally finite action of $\nnn$ on $L^0(\eta,\la)$ implies (W2). Now to verify (W3), we observe  the following fact: for any $\zzz\in Z(\ggg_0)$
			$$ \zzz\cd (\dis{\sum_{i} X_i \ot w_i}) =\dis{\sum_{i} \zzz\cd X_i \ot w_i}+  \dis{\sum_{i} X_i \ot \zzz\cd w_i}, \,$$
			for all $ \, X_i \in U(\mf g) _i, \, w_i \in L^0(\eta,\chi)$.
			Note that in the above sum $Z(\mf g_0).w_i$ is finite-dimensional, since $L^0(\eta,\la)$ is a module in the Whittaker category of $\mf{gl}(n)$. Also each $Z(\mf g_0).X_i$ is an finite-dimensional since $Z(\mf g_0).X_i \in U(\mf g)_i$ and $\dim (U(\mf g)_i) < \infty$. We have that $ Z(\ggg_0)\cd (\dis{\sum_{i} X_i \ot w_i})$ is finite-dimensional.  The axiom (W3) is confirmed.  This completes the proof of the part (1).
			
			For the part (2), we first decompose $\mcal M(\eta,\la)$ into the sum of eigen spaces with repect to $I_n:=\sum_{i=1}^n y_iD_i\in\hhh$. Then
			\begin{align}\label{eq: m chi eta}
				\mcal M(\eta,\la)= \dis{\bgop_{m \in \N} } \mcal M(\eta,\la)_m
			\end{align}
			where $\mcal M(\eta,\la)_m = U(\mf g_{\geq 1})_m \ot L^0(\eta,\la)$.
			Therefore it is clear that  $\mcal M(\eta,\la)_m $ is a $\mf g_0$-module for all $m \in \N$. Let $W$ be any non-zero proper submodule of $ \mcal M(\eta,\la)$. It is routine to prove that any submodules of $\mcal M(\eta,\la)$ is still a weight module with respect to the toral  Lie subalgebra $\bbf I_n$. Thus $W$ has an $I_n$-weight space decomposition:  $W= \dis{\bgop_{m \in \N} } W_m, $ where $W_m = \mcal M(\eta,\la)_m \cap W$.  Now we claim that $W_0=0$. Otherwise $W_0$ is a submodule of $L^0(\eta,\la)$ which is irreducible. This contradicts the assumption that $W$ is proper. Therefore the sum of all proper submodules is still proper and hence it has a unique maximal submodule.  Consequently, $\mcal L(\eta,\la)$ has  a unique  irreducible quotient.

			For the  part  (3),  we first consider (a)$\Rightarrow$(b).  Suppose $\mcal M(\eta,\la) \simeq \mcal M(\eta,\mu)$,  then their weight spaces of  $\bbf I_n$-zero weight in the sense of (\ref{eq: m chi eta}) must be isomorphic,  as $\ggg_0$-modules. This is to say,  $\mcal M(\eta,\la)_0 \simeq \mcal M(\eta,\mu)_0$ as $ \mf g_0$-modules. Hence $L^0(\eta,\la) \simeq L^0(\eta,\mu)$ as $\mf g_0$-modules.  Consequently,  by (1) we have  $\mcal L(\eta,\la) \simeq \mcal L(\eta,\mu)$.
			
			For (b)$\Rightarrow$(c),   suppose we already have that $\mcal L(\eta,\la) \simeq \mcal L(\eta,\mu)$. By the same arguments as above,  their weight spaces of  $\bbf I_n$-zero weight in the sense of (\ref{eq: m chi eta}) must be isomorphic. Hence  $L^0(\eta,\la) \simeq L^0(\eta,\mu)$ as $\mf g_0$-modules.
			Hence it follows from \cite[Prop. 2.1]{MS} that $\mcal W_\eta \bullet \la = \mcal W_\eta \bullet \mu$.

			For (c)$\Rightarrow$(a), by applying \cite[Prop. 2.1]{MS} again we have that  $\mcal W_\eta \bullet \la = \mcal W_\eta \bullet\mu$ implies $L^0(\eta,\la) \simeq L^0(\eta,\mu)$ as $\mf g_0$-modules.  Therefore,  by definition it is  directly deduced that   $\mcal M(\eta,\la) \simeq \mcal M(\eta,\mu)$.
			%
			
			The proof is completed. 
		\end{proof}

		\section{Decomposition of the Whittaker category and classification of  simple objects}
		Keep the notations as before.
		\subsection{Decomposition of the Whittaker category}		
		Assume that $\eta$ is a character of $\nnn^+$ and extends to the Lie superalgebra homomorphism
		\begin{align}\label{eq: eta}
			\eta: \sfn=\nnn^+\oplus\ggg_{-1}\longrightarrow \bbf
		\end{align}
		with sending $\ggg_{-1}$ to zero.
		Let us denote by $\mscr W_\eta$ the subcategory of $\mscr W$, whose objects are locally annihilated by some powers of $\ker \eta$ in (\ref{eq: eta}). Then we have the following theorem.	
		
		\begin{thm}\label{thm: cat decomp}
			\begin{enumerate}
				
				\item The category $\mscr W$ splits into a direct sum of subcategories according to the characters over $\nnn^+$ as below
				$$\mscr W= \dis{\bigoplus_{\eta\in \text{ch}(\nnn^+)}}\mscr W_\eta$$
				where $\text{ch}(\nnn^+)$ denotes the set of characters of $\nnn^+$ which extend to $N^*$ by taking zero at $\ggg_{-1}$.

				\item Furthermore, $\mscr W_\eta$ can be decomposed into a direct sum of subcategories $\mathscr{W}_{\eta,\chi}$ with $\chi$ ranging over the central characters of $Z(\ggg_0)$.
			\end{enumerate}			
		\end{thm}
		
		\begin{proof}  (1) Suppose $M$ is an arbitrarily given object in $\mscrw$. Still consider $\ker(\eta)\subset N$ for $\eta\in \text{ch}(\nnn^+)$ in the same sense as in (\ref{eq: eta}).
			Let
			$$M^\eta:=\{v\in M\mid  \ker(\eta)^{N_{v}}v=0 \}.$$
			Obviously, $M^\eta$ is a $U(\ggg)$-module. For this, we only need to show that $M^\eta$ is closed under $\ggg$-action. In order to check this, we first notice that
			$$y_{1}y_{2}\cdots y_{n}x_{0}=\sum\limits_{l=1}^{n}\sum\limits_{i_{1}<i_{2}<\cdots i_{l}} C_{y_{i_{1}},y_{i_{2}},\dots,y_{i_{l}},x_{0}}ady_{i_{1}}ady_{i_{2}}\cdots ady_{i_{l}}(x_{0})y_{1}\cdots\hat{y_{i_{1}}}\cdots \hat{y_{i_{l}}}\cdots y_{n},$$
			
			with homogeneous elements $x_{0}\in U(\ggg),y_{1},\cdots,y_{n}\in N$, and $C_{y_{i_{1}},y_{i_{2}},\dots,y_{i_{l}},x_{0}} \neq 0$
			
			Then it follows that $\forall v\in M^{\eta}, x_{0}\in U(\ggg)$, $\exists \, n$ large enough, such that $y_{1}y_{2}\cdots y_{n}x_{0}\cdot v=0$, for any homogeneous elements $y_{1},\dots,y_{n}\in \nnn^{+}$($\ggg_{-1}$ respectively) by applying Engel's Theorem on $\nnn^{+}$ and $\ggg_{-1}$ respectively. Hence $M^{\eta}$ is a $U(\ggg)$-module.
			
			Next, we will prove that $M=\bigoplus\limits_{\eta\in \ch(\nnn^{+})}M^{\eta}$, $\forall M\in \mscr W$. Notice that $\forall v\in M$, $W:=U(N)\cdot v$ is finite dimensional, then we can get the decomposition $W=\bigoplus\limits_{\eta\in \ch(N)}W_{\eta}$, where $W_{\eta}$ is the direct sum of all indecomposable components giving rise to $\eta$. We define $$W_{\eta}^{'}=\{v\in W|(x-\eta(x))^{N_{x}}\cdot v=0, \forall \, x\in N\}.$$
			
			It is clear that $W_{\eta}\subseteq W_{\eta}^{'}$. Now we need to prove $W_{\eta}^{'}\subseteq W_{\eta}$. Otherwise, if $W_{\eta}^{'}\cap\bigoplus\limits_{\mu\ne\eta}W_{\mu}\ne\{0\}$, there exists some $x\in N$, such that $\lambda(x)$ takes different value for all different character $\lambda$. Now we take some nonzero $v=\sum v_{\mu}\in W_{\eta}^{'}\cap\bigoplus\limits_{\mu\ne\eta}W_{\mu}$, with $v_{\mu}\in W_{\mu}$. There exists some positive integer $N_{\mu},N_{\eta}$, such that $(t-\mu(x))^{N_{\mu}}v_{\mu}=0,(t-\eta(x))^{N_{\eta}}v=0$.
			Since $(t-\eta(x))^{N_{\eta}}$ and $\Pi_{\mu\ne\eta}(t-\mu(x))^{N_{\mu}}$ are coprime, there exist $a(t),b(t)\in \mathbb{C}[t]$, such that $a(t)(t-\eta(x))^{N_{\eta}}+b(t)\Pi_{\mu\ne\eta}(t-\mu(x))^{N_{\mu}}=1$. Then it follows that $v=a(x)(x-\eta(x))^{N_{\eta}}\cdot v+b(x)\Pi_{\mu\ne\eta}(x-\mu(x))^{N_{\mu}}\cdot v=0$, which is a contradiction.
			Thus $W_{\eta}^{'}= W_{\eta}$, $M=\bigoplus\limits_{\eta\in ch\nnn^{+}}M_{\eta}^{'}$.
			Furthermore, it is easy to see that $M_{\eta}^{'}\subseteq M^{\eta}$ and $M_{\eta}\cap\sum\limits_{\mu\ne\eta}M_{\mu}={0}$. This leads to $M_{\eta}^{'}=M^{\eta}$, $M=\bigoplus\limits_{\eta\in \ch(\nnn^{+})}M^{\eta}$.

			(2) For any $M\in\mscr W$ and central character $\chi$ of $Z(\ggg_{0})$, we denote $M_{\chi}=\{v\in M|(z-\chi(z))^{n_{v}}\cdot v=0\}.$ Since $\forall v\in M$, $Z(\ggg_{0})\cdot v$ is finite dimensional, $Z(\ggg_{0})$ is nilpotent Lie algebra, similar to the proof in (1), we have $Z(\ggg_{0})\cdot v=\bigoplus\limits_{\chi}(Z(\ggg_{0})\cdot v)_{\chi}$, $\chi$ ranging over the central characters of $Z(\ggg_{0})$. It follows that $M=\bigoplus\limits_{\chi}M_{\chi}$. Then $M_{\eta}=\bigoplus\limits_{\chi}M_{\eta}\cap M_{\chi}$ and $\mscr W_\eta=\bigoplus\limits_{\chi}\mscr W_{\eta,\chi}$.

		\end{proof}
		
		

		\subsection{Irreducible objects of Whittaker category and their classifications}
		Due to Theorem \ref{thm: cat decomp}, the classification of simple objects for $\mscrw$ amounts to the one for $\mscrw_{\eta}$.
		
		\begin{thm}\label{thm: 3.2} The following statements hold.
			\begin{itemize}
				\item[(1)] Any simple objects of $\mscrw_\eta$ are isomorphic to some $\mcal{L}(\eta,\la)$ with $\lambda\in \hhh^*$.
				
				\item[(2)] Any simple objects $\mcal{L}(\eta,\la)$ and $\mcal{L}(\eta,\mu)$ are isomorphic if and only if $\mathcal{W}_\eta\bullet\la=\mathcal{W}_\eta\bullet\mu$.
				
			\end{itemize}
		\end{thm}
		\begin{proof}
			(1) Let $L$ be any given simple object in $\mscr W$. By Lemma \ref{lem: abelian}, $L$ is an irreducible $U(\ggg)$-module. Denote by $\text{Ann}_{\ggg_{-1}}(L)$  the annihilated  space in $L$ by $\ggg_{-1}$.   We first claim that $\text{Ann}_{\ggg_{-1}}(L)$ is nonzero. This is because $L$ is nonzero, we can take a nonzero vector $v\in L$. By definition, there exits some positive integer $N$ such that $\ggg_{-1}^{N}v=0$. We can take the least one $N_v$ such that $\ggg_{-1}^{N_v}v=0$ but $\ggg_{-1}^{N_v-1}v\ne 0$.  Thus, the nonzero vectors from $\ggg_{-1}^{N_v-1}$ belong to $\text{Ann}_{\ggg_{-1}}(L)$. The claim is proved. Consequently, $\text{Ann}_{\ggg_{-1}}(L)$ is a nonzero $\ggg_0$-module.

			Now consider  $W:= U(\ggg_0)v $ for some nonzero vector $v \in \text{Ann}_{\ggg_{-1}}(L)$. It is clear that $W$ is a $\ggg_0$-Whittaker module (see \cite{MS}). Then by \cite[Theorem 2.6]{MS} we get a composition series for the $\ggg_0$-module $W$. In particular, $W$ contains an irreducible submodule over $\ggg_0$, which is of the form $L^0(\eta,\la)$ for some $\eta\in\text{ch}(\nnn^+)$ and $\la\in \hhh^*$  (see \cite[Therem 2.6]{MS}). { Now by the choice of $v$,   $L^0(\eta,\la)$ is a module for $P=\ggg_{-1} \oplus \ggg_0$ with trivial action of $\ggg_{-1}$ on it}. Thus $L$ is an irreducible quotient of the induced module $\mcal M(\eta,\la)=U(\mf g ) \ot_{U(P)} L^0(\eta,\la)$. Now it follows by Proposition \ref{Thm: unique irr quo} that $L$ is isomorphic to $\mcal L(\eta,\la)$.  This is desired.

			(2) This  follows from Proposition \ref{Thm: unique irr quo}(3).
		\end{proof}

		\subsection{Finite length property only for $W(2)$}\label{sec: finte length}
		\begin{prop}\label{lem: w2fintelength}
			\begin{itemize}
				\item[(1)] 	Every object in $\mscr W$ has finite length for $W(2)$.
				\item[(2)] $\mcal M(\eta,\la)\in \mscr W_{\eta,\chi}$ has infinite length for $W(n)$ with central character $\chi=\chi_{\la}$, when $n >2$.
			\end{itemize}
			
		\end{prop}
		
		\begin{proof}
			Let $M$ be any object in $\mscr W$.  By Theorem \ref{thm: cat decomp}, we can suppose that  $U(\mf g)$-module $M$ is decomposed into a direct  sum of finitely many $M_{\eta,\chi}$ for some different pairs $(\eta,\chi)$ with  $M_{\eta,\chi}\in \mscr W_{\eta,\chi}$.
			Without loss of generality, we suppose $M\in \mscr W_{\eta,\chi}$.

			By definition, $M$ is finitely generated over $U(\ggg)$. Suppose $M=\sum_{i=1}^m U(\ggg)m_i$. According to the property of  locally-finiteness over $U(\nnn+\ggg_{-1})$ associated with $\chi$,
			$V:=\sum\limits_{i=1}^m U(\nnn+\ggg_{-1})m_i$ is finite-dimensional,  with $\mf g_{-1}$-nilpotent action. Consequently, $M\in \mscr W_{\eta,\chi}$ admits a $U(\sfp)$-submodule $U(\sfp)V$, which is easily proved to be a Whittaker module over $U(\ggg_0)$.   By \cite[Theorem 2.6(1)]{MS}, $U(\sfp)V$ has finite length over $U(\ggg_0)$ and each composition factor of $U(\sfp)V$ as a $U(\sfp)$-module is a simple Whittaker $\ggg_0$-module associated with the pair $(\eta,\chi)$ with trivial $\ggg_{-1}$-action.
			With this observation and taking a fact into account that  the $U(\ggg)\otimes_{U(\sfp)} -$ is an exact functor from the Whittaker category of $\sfp$ with trivial $\ggg_{-1}$-action (associated with the pair $(\eta,\chi)$), to the category $\mscr W_{\eta,\chi}$, we only need to show that  $\mcal M(\eta,\la)$ has finite length over $U(\ggg) $ with $\chi=\chi_{\la}$.
			%
			%
			
			For this, we note that
			$$\mcal M(\eta,\la)=\dis{\bgop_{m =0}^{2} } U(\mf g_{\geq 1})_m \ot L^0(\eta,\la),  $$
			since $U(\mf g_{\geq 1})_m=0 $ for all $ m \geq 3$. Furthermore we note that $U(\mf g_{\geq 1})_m$ is finite dimensional module for $\mf g_0$ for all $m$. Hence $\mcal M(\eta,\la)$ is a finite direct sum of Whittaker modules for { $\mf g_0$} by Theorem H of \cite{K}. This proves that $\mcal M(\eta,\la)$ has finite length as a $\mf g_0$-module and hence have finite length as $W(2)$-module, which completes the proof of (1).
			
			To prove (2), we assume that $\mcal M(\eta,\la)$ has finite length. Then $\mcal M(\eta,\la)$ is an Artinian module. Now note that for each $k \in \N$, $M^k= \dis{\bgop_{i \geq k}\mcal M(\eta,\la)_i   }$ is a $\mf g_{\geq 0}$-submodule of $\mcal M(\eta,\la)$, where $\mcal M(\eta,\la)_i  =U(\mf g_{\geq 1})_i \ot L^0(\eta,\la)$.  Furthermore, it is easy to observe that $\mcal M(\eta,\la)=M^0 \supset M^1 \supset \dots $ is a chain of $\mf g_{\geq 0}$- submodules with $M^k/M^{k+1} \neq 0$, since $U(\mf g_{\geq 1})_i \neq 0 $ for all $i$. Therefore it induces a chain of $\mf g$-submodules and each of whose subquotients are non-zero, i.e $\mcal M(\eta,\la)$ is not artinian, a contradiction. This completes the proof.
		\end{proof}


		
		\section{Generalized Backelin functor from the category $\mathcal{O}$ to $\mscrw$ and applications in the case $W(2)$}\label{sec: sec three}
		\subsection{}		Following Backelin \cite{Back} we construct a generalized Backelin functor for $W(n)$ and study some properties of this functor.  Then we give some application in the case of $W(2)$.  
		\subsubsection{The parabolic BGG category $\mathcal{O}$}

		We first recall the parabolic BGG category for $\ggg$ introduced in \cite{DSY}
		\begin{defn}\label{category O} The parabolic BGG category $\calo$ is a full category of $U(\ggg)\hmod$,  whose objects $M$  are required to be finitely-generated over $U(\ggg)$ satisfying the following three axioms:
			\begin{itemize}
				\item[(O1)] $M$ is an admissible $\mathbb{Z}$-graded $\ggg$-module, i.e., $M=\bigoplus\limits_{i\in\mathbb{Z}} M_{i}$ with $\dim M_{i}<+\infty$, and
				$\ggg_{i}M_{j}\subseteq M_{i+j},\forall\,i,j$.
				\item[(O2)] $M$ is locally finite as a $\sfp$-module where $\sfp=\ggg_{\leq 0}:=\ggg_{{-1}}\oplus\ggg_{0}$.
				{ \item[(O3)] $M$ is $\hhh$-semisimple, i.e., $M$ is a weight module: $M=\bigoplus_{\lambda\in\hhh^*}M_{\lambda}$.}
			\end{itemize}
			The morphisms in $\calo$ are the $\ggg$-module morphisms that preserve the $\mathbb{Z}$-grading, i.e.,
			$$\Hom_{\calo}(M, N)=\{f\in\Hom_{U(\ggg)}(M, N)\mid f(M_{i})\subseteq N_{i},\,\forall\,i\in\mathbb{Z}\},\,\forall\,M,N\in\calo.$$
		\end{defn}

		\subsubsection{An automorphism $\bar\sigma$ of $W(n)$}	
		Let $\sigma=\prod_{j=1}^{[{n\over2}]} (j,n+1-j) $ be an element of the symmetric group $\mathfrak{S}_n$ where $[{n\over 2}]$ denotes the greatest integer not bigger than $n\over 2$. Now define an automorphism of $\mcal R$  via the following transform on the set of  generators $\{y_i\mid i=1,\ldots,n\}$
		$$ y_i \to y_{n+1-i} , \, \, \forall i \in [1,n]:=\{1,2,\ldots,n\} $$
		and denote it by the same notation $\sigma$.
		Then $\sigma$ induces an automorphism $\bar{\sigma}$ of $W(n)$ via
		$$\bar \sigma(X)= \sigma \circ X \circ \sigma^{-1}, \,\,\,\, \forall X \in W(n).$$
		By a straightforward check it is easily  known that  for $f\in \calr$
		$$\bar \sigma(fD_j)=\sigma(f)\bar\sigma(D_j)= \sigma(f)D_{\sigma(j)}.$$
		This implies that $\bar \sigma(\nnn^+)=\nnn^-$, $\bar \sigma(\nnn^-)=\nnn^+$ and $\bar \sigma(\mf h)=\mf h$.
		
		On the other side, $\sigma\in \mathfrak{S}_n$ gives rise to a linear automorphism of $\mf h^*$. Precisely, for $\{\epsilon_1, \cdots,\epsilon_n \}$ the basis of $\mf h^*$ which are subject to $\epsilon_i(y_jD_j)=\de_{ij}$, $i,j=1,\ldots,n$,   we define $\sigma(\epsilon_i)=\epsilon_{\sigma(i)}$.

		\subsubsection{The $\bar\sigma$-twisted duals and a functor from $\mathcal{O}$ to $\mscrw_\eta$} 	
		\begin{defn} Define a functor $\bfB$ from $\calo$ to $U(\mf g)$-mod, which sends  $M = \bigoplus_{\la \in \mathfrak{h}^*} M_\la$ in $\mathcal{O}$ to $\overline{M} := \prod_{\la \in \mathfrak{h}^*} M_\la$ (the complete of $M$) in $U(\mf g)$-mod.
		\end{defn}

		
		Furthermore,  if $M=\sum_{\lambda\in\hhh^*}M_\lambda \in \mathcal{O}$, then we define $M^\vee = \bigoplus_{\la \in \mathfrak{h}^*} (M_\la)^*$   with $\ggg$-action on $M^\vee$: for $X\in \ggg$ and $\phi\in M^\vee$,
		\begin{align}\label{eq: twist dual}
			X.\phi=\phi \circ \bar\sigma(X).
		\end{align}
		It's easy to verify that $M^\vee$ still lies in the category $\mathcal{O}$ and the canonical homomorphism $M \to M^{\vee \vee}$ is an isomorphism, which can be extended to an isomorphism of $U(\mathfrak{g})$-modules
		\[
		\overline{M} \to M^{\vee'},
		\]
		where for any $N\in U(\ggg)\hmod$, $N'$ denotes a new $U(\ggg)$-module whose underlying space is the linear dual of $N$, and whose $U(\ggg)$-module structure is defined as in (\ref{eq: twist dual}). This module  $N'$ is called the $\bar\sigma$-twisted dual of $N$.

		\subsubsection{Generalized Backeline functor}  Generally, for $M \in U(\ggg)\hmod$, we can define
		\[
		\Gamma_\eta(M) := \{ m \in M :  (\operatorname{Ker} \eta)^k \cdot m = 0, \,\, \text{for some} \, k \} \in {U}(\mathfrak{g})\text{-mod};
		\]
		and for $M \in \mathcal{O}$, we put
		\[
		\overline{\Gamma}_\eta(M) := \{ m \in \overline{M} : (\operatorname{Ker} \eta)^k \cdot m = 0, \,\, \text{for some} \, k  \} \in {U}(\mathfrak{g})\text{-mod},
		\]
		here $\operatorname{Ker} \eta$ denotes the kernel of $\eta$ in ${U}(\nnn)$.  \\

		\begin{lem}\label{exact fun}
			$\overline{\Gamma}_\eta$ defines an exact functor from $\mathcal{O}$ to $\mscr W_\eta$.
			
		\end{lem}
		\begin{proof}
			The proof is parallel to that of \cite[Lemma 3.2]{Back} with  aid of $\bar \sigma$.
		\end{proof}

		\subsection{ Standard module in $\mcal O$:}	Let $\Lambda^+$ be the set of dominant integral weights relative to the standard Borel subalgebra $\mathfrak{b}:=\hhh+\nnn^+$ of $\ggg_{0}$. Then finite-dimensional irreducible $\ggg_{0}$-modules are parameterized by ${\Lambda^+}\times\mathbb{Z}$. For any $\lambda\in {\Lambda^+}$, let ${}^dL^0(\lambda)$ be the simple $\ggg_{0}$-module concentrated in a single degree $d$ with the highest weight $\lambda$.  Set ${}^d\Delta(\lambda)=U(\ggg)\otimes_{U(\sfp)}{}^dL^0(\lambda)$,  where ${}^dL^0(\lambda)$ is regarded as a $\sfp$-module with trivial $\ggg_{{-1}}$-action. Then ${}^d\Delta(\lambda)$ is a standard module of depth $d$, and
		$\{{}^d\Delta(\lambda)\mid \lambda\in {\Lambda^+},d\in\mathbb{Z}\}$ constitutes a class of so-called standard modules of depth $d$ for $\ggg$ in the usual sense. We forget the depth and denote $M(\la)={}^0\Delta(\lambda).$
		
		\subsection{Backelin functor for $\mf g_0$ and its application in the case $W(2)$} Logically continuing the arguments in \S\ref{sec: finte length}, we investigate the composition series of  the standard modules when the finite-length composition series of objects in $\mscrw$ happen. This is just the case $W(2)$.
		
		In the following, we set $\ggg=W(2)$.  	Let $\Res^{\mf g}_{\mf g_0}:\mf g$-mod $\to$ $\mf g_{0}$-mod be the restriction functor and $\mscr W^{\mf g_0}$ be the category of Whittaker module for $\mf g_0$. Then by \cite{Back}, the Backelin functor $\Gamma_\eta^0:\mcal O ^{\mf g_0} \to \mscr W^{\mf g_0}$, for $\eta \in \ch(\mbf n)$ is exact.
		
		Let $W(2)=\nnn^- \op \hhh \op \nnn$ be the triangular decomposition with $\nnn=\bbf y_{1}D_{2}+\bbf y_{1}y_{2}D_{1}+\bbf y_{1}y_{2}D_{2}$, $\nnn^{-}=\bbf y_{2}D_{1}+\bbf D_{1}+\bbf D_{2}$ and $\hhh=\bbf y_{1}D_{1}+\bbf y_{2}D_{2}$. For any $\mf g$-module $M$ in category $\mcal O$, we denote $\ov\Gamma_{\eta}: \mcal O \to \mscr W$ by sending $M$ to be the submodule consisting of vectors in $\ov{M}$ annihilated by some power of $x-\eta(x)$, $\forall x\in \nnn_{\bar{0}}=\bbf y_{1}D_{2}$. For convenience, we call a weight $\mu$ to be $\mbf n_\eta$ anti-dominant, if $\mu|_{\nnn_{\eta}}\ne 0$ and $\mu$ is anti-dominant.
		
		\begin{prop}
			Let  $\mf g=W(2)$, then we have the following.
			\begin{itemize}
				\item[(1)] 	\begin{equation}
					\ov\Gamma_\eta(M(\la))\cong
					\begin{cases}
						\mcal M(\sigma\la, \eta), \,\, \text{if}\,  \sigma \la  \,\,\text{is} \,\, \mbf n_\eta \text{ anti-domiant}\\
						0 \,\, otherwise.
					\end{cases}
				\end{equation}
				\item[(2)] 	\begin{equation}
					\ov\Gamma_\eta(L(\la))\cong
					\begin{cases}
						\mcal L (\sigma\la, \eta), \,\, \text{if}\,  \sigma \la  \,\,\text{is} \,\, \mbf n_\eta \text{ anti-domiant}\\
						0 \,\, otherwise.
					\end{cases}
				\end{equation}
				
				\item[(3)] $[M(\la):L(\mu)]=[\mcal M(\sigma \la,\eta):\mcal L(\sigma \mu, \eta)]$.

			\end{itemize}
			
		\end{prop}	
		
		\begin{proof}
			First assume that $\sigma \la$ is $\mbf n_\eta$ anti domiant. By PBW theorem we see that as a vector space $\ov{M(\la)} = U(\mf g_{\geq 1})\ot \ov{L^0(\sigma\la )} $, since for $W(2)$, the automorphism $\bar\sigma$ behaves like the anti automprphism of $\mf g_0$ except on $\mf h$.  Hence $\ov\Gamma_\eta(M(\la)) $ consists all vectors $v$ of $ U(\mf g_{\geq 1})\ot \ov{L^0(\sigma\la)}$ such that $x -\eta(x)$ acts nilpotently on $v$ for all $x \in \bfn$. Therefore there is an embedding of  $U(\mf g_{\geq 1})\ot \Gamma_\eta^0({L^0(\sigma\la)}) \hookrightarrow \ov\Gamma_\eta(M(\la)) $ induced by the embedding $\Gamma_\eta^0({L^0(\sigma\la)}) \hookrightarrow \ov{L^0(\sigma\la)}$. Now by \cite{Back}, we have $\Gamma_\eta^0({L^0(\sigma\la)})  = L^0(\sigma\la, \eta)$ and hence $U(\mf g_{\geq 1})\ot L^0(\sigma\la, \eta) \hookrightarrow \ov\Gamma_\eta(M(\la)).$ Now we prove that the induced embedding is onto. To prove this it is sufficient to show that as $\mf g_0$ module both of $U(\mf g_{\geq 1})\ot L^0(\sigma\la, \eta)$ and $\ov\Gamma_\eta(M(\la))$ have same number of composition factors. Since both of them are Whittaker $\mf g_0$ module they have finite length as $\mf g_0$ module. Now the isomorphism follows from the following observation. \\
			$\Res^{\mf g}_{\mf g_0} \ov{\Gamma}_\eta(M(\la)) =  \Res^{\mf g}_{\mf g_0} \ov{\Gamma}_\eta( U(\mf g_{\geq 1})\ot {L^0(\la))}=\Res^{\mf g}_{\mf g_0}( U(\mf g_{\geq 1})\ot  {\Gamma}_\eta^0{L^0(\sigma\la))}=\Res^{\mf g}_{\mf g_0}( U(\mf g_{\geq 1})\ot {L^0(\sigma\la,\eta))}=\Res^{\mf g}_{\mf g_0}(\mcal M(\sigma\la, \eta))$.\\
			Thus from the above we have as a $\mf g_0 $ module $ \ov\Gamma_\eta(M(\la)) \cong M(\sigma\la, \eta)$. Now consider the action of $I_{n \times n} \in \mf g_0$ on $\ov\Gamma_\eta(M(\la))$, which acts semi-simply on it due to the isomorphism. Further it is easy to see that this action induces a $\Z$ gradation with zeroth graded component as  $L^0(\sigma\la, \eta)$. In  particular $\mf g_{-1}$ acts trivially on $L^0(\sigma\la, \eta)$. Now by Frobenius reciprocity we have the desired isomorphism.

			Also it is clear from the above proof and  Proposition 6.9(3) of \cite{Back}  that when $\sigma \la $ is not $\mbf n_\eta$ anti dominant, then $\ov\Gamma_\eta(M(\la))=0$. This completes the proof of (1).\\
			To prove (2), note that $L(\la)$ is an irreducible module in Category $\mcal O$ for $W(2)$. Also note that $W(2) \cong \mf{osp}(2,2)$ (see \cite{CW12}, Chapter 1). Further it is easy to see that category $\mcal O $ for $W(2)$ is a subcategory of category $\mcal O$ for $\mf{osp}(2,2)$. Therefore one can apply $\Gamma_\eta$ on modules of category $\mcal O$ for $W(2)$.  Hence we get that $\ov \Gamma_\eta(L(\la))= \Gamma_\eta(L(\sigma\la))$ this equality follows from the fact that action of $\ov \Gamma_\eta$ and $ \Gamma_\eta$ only differs on $\mf h$. Further category of Whittaker modules for $W(2)$ and category of whittaker module for $\mf{osp}(2,2)$ coincides (see \cite{CWC}, for the definition of Whittaker module of $\mf{osp}(2,2)$). Therefore we have the result from Theorem 9 and Theorem 20 of \cite{CWC}. Now (3) follows with the help of Lemma \ref{exact fun}.

		\end{proof}

		\subsection{} We continue to  consider the Lie super algebra $\mf{osp}(2,2)$ for Whittaker representations of $W(2)$.
		$M(\la)$ can be considered as the Kac module for $\mf{osp}(2,2)$. Now we recall the root system for $\mf{osp}(2,2)$. Let $\Delta_{\bar 0}$ and $\Delta_{\bar 1}$ be the sets of even and odd roots of $\mf{osp}(2,2)$. Let $h_1=E_{11}-E_{22}$ and $h_1'=E_{33}-E_{44}$. Then $ \rm{span}\{h_1,h_1'\}$ forms a basis for the Cartan subalgebra of $\mf{osp}(2,2)$. Define the dual basis $\epsilon, \delta$ of $h^*$ as
		$$ \epsilon(h_1)=1, \, \,\, \epsilon(h_1')=0, \, \text{and}  $$
		$$ \delta(h_1)=0, \,\,\,\, \delta(h_1')=1.  $$
		Now define a blinear form on $h^*$ as
		$$(\epsilon, \epsilon)=1, \,\,\, (\delta,\delta)=-1,\,\, (\epsilon, \delta)=(\delta, \epsilon)=0.  $$
		Then we get that $$\Delta_{\bar 0}=\{\pm 2\delta\} \,\,\, \text{and} \,\,\, \Delta_{\bar 1}=\{ \pm \epsilon \pm \delta \}.$$
		Further, $\Delta_{\bar 0}^+={2\delta}$ and $\Delta_{\bar 1}^+=\{ \epsilon -\delta, \epsilon+\delta \}$ are the sets of even and odd positive roots of $\mf{osp}(2,2)$. Hence
		
		$$\rho= \frac{1}{2} \sum_{\al \in \Delta_{\bar 0}^+ }\al -\frac{1}{2}\sum_{\al \in \Delta_{\bar 1}^+ } \al= \delta -\epsilon.$$
		A weight $\la \in h^*$ is said to be atypical for $\mf{osp}(2,2)$ if there exists a root $\ga$ in $\Delta_{\bar 1}^+$ such that $(\rho+\la, \ga)=0$.\\
		Let $\la=a\epsilon_1 +b \epsilon_2$ be a dominant integral weight of $\mf{gl}(2)$. Since $\mf{osp}(2,2)_0 \cong \mf{gl}(2)$, we can treat $\la$ as a dominant integral weight for $\mf{osp}(2,2)$. In particular, one can see that $\la= a\epsilon +b \delta$. Now $\la$ will be atypical for $\mf{osp}(2,2)$ if
		$$(\rho+ \la,\epsilon -\de)=0, \,\, \text{or} \,\, (\rho+ \la,\epsilon +\de) =0, \, i.e, $$
		$$ b=-a,  \,\, \text{or} \,\, b=a-2, i.e, \,\,  \la= a\epsilon -a \delta, \, a \geq 0, \,\, \text{or} \,\, \la= a\epsilon +(a-2) \delta ,  \, \, a \in \Z. $$
		
		The following proposition follows from Theorem 3.3 of \cite{SZ}. 		
		
		\begin{prop}
			Let $\la $ be a dominant integral weight of $\mf{osp}(2,2)$.
			\begin{itemize}
				\item[(1)]  Kac module $K(\la)$ of $\mf{osp}(2,2)$ has composition series of length 1 or 2, according as $\la$ is typical or atypical.
				\item[(2)] There exists a unique positive atypical root $\ga$ of $\la$.
				\item[(3)] Let $k$ be the smallest positive integer such that  $\mu=\la -k \ga$ and $(\mu+\rho,\al)\neq 0$ for all $ \al \in \Delta_{\bar 0}^+$. Then there exists a unique element $\omega$ in Weyl group of $\mf{osp}(2,2)$ such that $\la^1=\omega(\mu+\rho)-\rho$ is dominant integral for $\mf{osp}(2,2)$.
				\item[(4)] $[K(\la),L(\la)]=1$ and $[K(\la),L(\la^1)]=1,$ when $\la $ is atypical.
				\item[(5)] $[K(\la),L(\la)]=1,$ when $\la$ is typical.
			\end{itemize}

		\end{prop}
		Now the following Theorem follows from Proposition $4.5$ and above discussion.
		
		\begin{prop} (Number of composition factor of $\mcal M(\eta,\la)$ for $W(2)$) Let $ \la \in \mf h^*$ be dominant integral and $\la^1$ be as in Proposition $4.5$ Then
			\begin{itemize}
				\item[(1)] $[\mcal M(\eta,\sigma\lambda):\mcal L(\eta,\sigma\lambda)]=[\mcal M(\eta,\sigma\lambda):\mcal L(\eta, \sigma\la^1)]=1,$
				when $\la=a \epsilon - a\delta , \,  a \geq 0 $ or $\la = a\epsilon + (a-2) \delta, \, a \in \Z.$
				\item[(2)] $[\mcal M(\eta,\la):\mcal L(\eta,\la)]=1 ,$ otherwise
			\end{itemize}
			
		\end{prop}

		\section{Finite $W$-superalgebras and generalized Skryabin equivalences}\label{sec: second part}
		
		\subsection{Finite $W$-superalgebras}
	For notational convenience, denote $\mf g_0=\bfg$. 	For any character $\eta:\nnn\rightarrow \bbf$, $\eta$ determines a subset $I$ of $\Delta$ which consists of $\alpha\in \Delta$ with  $\eta(\bfg_\alpha)\ne0$. Let $\bfg_\eta$ be the Levi subalgebra corresponding to $I$ with triangular decomposition $\bfg_\eta=\nnn_\eta^-\oplus \hhh \oplus \nnn_\eta$.

		Now we define
		$$ Q(\eta):= U(\ggg)\otimes_{U(\nnn_\eta)}\bbf_\eta,$$
		 where $\bbf_\eta$ is a one dimensional representation of $\nnn_\eta$.
		
		and consider the algebra
		$$U(\ggg,\eta):=\text{End}_\ggg(Q(\eta))^{\text{op}}.$$

		Let $I_{\eta}$ be the left ideal of $U(\ggg)$ generated by the set $\{ X-\eta(X):  X\in \bf n_{\eta} \}$. Then we have the following proposition.

		\begin{prop}\label{iso Q(eta,chi)}

			As left $U(\ggg)$-module, we have $Q(\eta)\simeq U(\ggg)\slash I_{\eta}$.

		\end{prop}
		
		\begin{proof}
			In order to get the isomorphism, we consider the following $U(\ggg)$-module homomorphism
			$$\phi: U(\ggg)\longrightarrow U(\ggg) \otimes_{U(\bf n_{\eta})}\bbf_{\eta}$$
			$$u\longmapsto u\otimes \mathbf{1}_{\eta}$$
			It is easy to see that $\phi$ is surjective and now we will determine $\ker\phi$.
			
			Since $U(\ggg)=U(\ggg)\bf n_{\eta}\bigoplus U(\ggg_{>0}+\ggg_{-1}+ \nnn_{\eta}^{-}+ \hhh)$, we can assume that $u=u_{1}+u_{2}$, where $u_{1}\in U(\ggg)\bf n_{\eta}$ and $u_{2}\in U(\ggg_{>0}+\ggg_{-1}+ \nnn_{\eta}^{-}+ \hhh)$. Naturally we can get $u_{1}=\sum\limits_{i}u_{1}^{(i)}{ n}_{\eta}^{(i)},u_{2}=\sum\limits_{j}u_{2}^{(j)}$ for all $u_{1}^{(i)},u_{2}^{(j)}\in U(\ggg_{>0}+\ggg_{-1}+ \nnn_{\eta}^{-}+ \hhh)$ and $n_{\eta}^{(i)}\in U(\bf n_{\eta})$. So it follows that
			$$u\otimes \mathbf{1}_{\eta}=(u_{1}+u_{2})\otimes \mathbf{1}_{\eta}=\sum\limits_{i}u_{1}^{(i)}\otimes n_{\eta}^{(i)}\mathbf{1}_{\eta}+\sum\limits_{j}u_{2}^{(j)}\otimes \mathbf{1}_{\eta}.$$
			This implies that $u\otimes \mathbf{1}_{\eta}$ can be expressed as the form $\sum\limits_{k}u^{(k)}\otimes X^{k}\mathbf{1}_{\eta}$, where $\{u^{(k)}\}$ are right $U(\bf n_{\eta})$-linearly independent and $X^{k}\in U(\bf n_{\eta})$. Hence $u\otimes \mathbf{1}_{\eta}=\sum\limits_{k}u^{(k)}\otimes X^{k}\mathbf{1}_{\eta}=0$ implies $X^{k}1_\eta=0$. Now it follows from the definition of $\eta$ that $ X^k \in I_\eta$.\\
			$u\otimes \mathbf{1}_{\eta}=\sum\limits_{k}u^{(k)}\otimes X^{k}\mathbf{1}_{\eta}=\sum\limits_{k}u^{(k)}X^{k}\otimes \mathbf{1}_{\eta}$ implies $u=\sum\limits_{k}u^{(k)}X^{k}\in I_{\eta}$. Hence  we have $\ker\phi=I_{\eta},$ this completes the proof.\\

			
		\end{proof}

		\subsection{Generalized Skryabin's equivalences }		
		In this part we consider a full subcategory $\mscr W_{\eta}'$ of $U(\ggg)\text{-mod}$ whose objects have locally nilpotent action of $\{X-\eta(X):X \in {\bf n}_\eta\}$. One can see that $\mscr W_{\eta}$ is the subcategory of $\mscr W_{\eta}'$. This $\mscr W_\eta'$ is the category of  so-called weakened Whittaker modules mentioned in the beginning of the paper.

		For  any $M \in \mathscr{W}^{'}_{\eta}$ we define,
		$$Wh_{\eta}(M)=\{ m \in M: X.m=\eta(X)m, \, \forall \, X \in {\bf n}_\eta\}.$$
By Proposition 5.1, one can define a module action of $U(\mf g,\eta)$ on $Wh_\eta(M)$. Now we have the following generalized Skryabin's equivalence, by following \cite{P}.
		\begin{thm}\label{thm: skryabin eq} There is a generalized Skryabin's equivalence between the category $U(\ggg,\eta)\text{-mod}$ and the category $\mathscr{W}^{'}_{\eta}$.
			
		\end{thm}	
		\begin{proof}
			Note that for $\eta=0$, $Q(\eta)=U(\ggg)$ and $U(g,\eta)=End_{\ggg}(U(\ggg))^{op}\simeq U(\ggg)$. Then it is easy to see that both $\mscr W^{'}_{0}$ and $U(\ggg,\eta)\text{-mod}$ are the category of $U(\ggg)$ modules. Now we only need to consider the case of $\eta\ne0$.
			
			In order to get the equivalence between two categories, we define the following functors:
			
			$$F: \mathscr{W}^{'}_{\eta}\longrightarrow U(\ggg,\eta) \text{-mod}$$
			$$M\longmapsto Wh_{\eta}(M), \,\,\, $$
			
		$$G: U(\ggg,\eta)\text{-mod} \longrightarrow \mathscr{W}^{'}_{\eta}$$
			$$V\longmapsto Q(\eta)\otimes_{U(\ggg,\eta)}V$$
			
			It is easy to check the well-definedness of $F$ and $G$ by Proposition 5.1.
			
			{\bf Case I:} Let $\eta $ be regular, i.e $\eta(x_k\pr_{k+1}) \neq 0 $ for $1 \leq k \leq n-1$.  Then $Y_{k}=x_{k}\pr_{k+1}$, $k=1,\dots,n-1$ forms a basis for $n_{\eta}$. Set $X_{k}=\frac{\sum\limits_{i=k+1}^{n}x_{i}\pr_{i}}{\eta(Y_{k})}$, $k=1,\dots,n-1$. Then it follows that $\eta([Y_{i},X_{j}])=\delta_{ij}$.
			
			For ${\bf a}=(a_1, \cdots, a_{n-1}) \in \N^{n-1},$ set	$$X^{\mathbf{a}}=X_{1}^{a_{1}}\cdots X_{n-1}^{a_{n-1}}\in U(\ggg), \,\,\, \text{and}$$ $$u_{\mathbf{a}}=(Y_{1}-\eta(Y_{1}))^{a_{1}}\cdots(Y_{n-1}-\eta(Y_{n-1}))^{a_{n-1}}\in U(\bf n_{\eta}).$$ Also let $|\mathbf{a}|=\sum\limits_{i=1}^{n-1}a_{i}$. We define an ordering on $\mathbb{N}^{n-1}$:
			for $\mathbf{a},\mathbf{b}\in \mathbb{N}^{n-1}$, define
			$$\mathbf{a}<\mathbf{b}$$
			if and only if
			$$|\mathbf{a}|<|\mathbf{b}|$$ or
			$$ \text{if} \,\,\,|\mathbf{a}|=|\mathbf{b}|$$ with
			$$a_{j}=b_{j},(j=1,2, \cdots,k-1 ) , \,\,\text{then} \,\, a_{k}<b_{k}. $$
			Let $J_{\eta}$ be the two sided ideal of $U(\bf n)$ generated by $X-\eta(X)$, for all $X \in \bf n_\eta$. Note that elements $u_{\bf a}$, for all ${\bf a} \in \N^{n-1}$ froms a basis for $U(\bf n_\eta)$, further $u_{\bf a} \in J_\eta^{|\bf a|}$.
			
			Denote by $I_{\mathbf{a}}$ the linear span of $u_{\mathbf{b}}$ with $\mathbf{b}>\mathbf{a}$. It is easy to see that $I_{\mathbf{a}}$ is an ideal of $U(\bf n_{\eta})$ and $J_\eta^r \subseteq I_{\bf a}$ for some $r >0$. \\
			Let $E_{V,{\eta}}=\{f\in \Hom_{\mathbb F}(U({\bf n}_\eta),V)|f(J_{\eta}^{m})=0, \, \text{ for some }  m>0\}$, then $E_{V,\eta}$ consists of all linear maps $f: U(\bf n_{\eta})\longrightarrow V$ such that $f(I_{\mathbf{a}})=0$, for some $\mathbf{a}$.\\
			
			For any $M\in \mathscr{W}^{'}_{\eta}$, we take $V=Wh_{\eta}(M)$. Now for any given $\mathbf{a},\mathbf{b}\in \mathbb{N}^{n-1}$ and $v\in V$, one can observe that $u_{\mathbf{b}}X^{\mathbf{a}}v=0$, whenever $\mathbf{b}>\mathbf{a}$ and $u_{\mathbf{a}}X^{\mathbf{a}}v=cv$ for some nonzero $c$. Take any linear map $\pi: M\longrightarrow V$ with $\pi|_{V}=id_{V}$ and define a $U(\bf n_{\eta})$-module map $\phi: M\longrightarrow E_{V,\eta}$ by sending $m\in M$ to $\phi(m): u\longmapsto \pi(um)$, for $u\in U(\bf n_{\eta})$. Note that $\phi$ is well defined due to that fact that $M \in \mathscr{W}^{'}_{\eta}$. Also it is eassy to see that $\phi$ is a $\bf n_\eta$-module map.
			
			Moreover, for any $f \in E_{V,\eta}$, we know that $f(I_{\bf a})=0$. Hence $f$ is determined by its action on $u_{\bf b}$ with $\bf b \leq \bf a$. Assume that $f(u_{\bf b})=v_{\bf b}$ for all $\bf b \leq \bf a.$ Then it is easy to see that $\phi(\sum_{{\bf b \leq \bf a}}X_{\bf b}v_b)$ is the the preimage of $f$ under $\phi$, i.e, $\phi$ is surjective. Again, $\ker\phi\cap V=0$ implies that $\phi$ is injective, since $\ker \phi$ is a $\bf n_\eta$-submodule of $M$. Then it follows that $\phi$ is an isomorphism and any element $m\in M$ can be uniquely expressed as the finite sum $\sum X^{\mathbf{a}}v_{\mathbf{a}}$.
			
			Let $M^{\mathbf{a}}=\{m\in M| I_{\mathbf{a}}.m=0\}$. Now for any $\mathbf{a}$, we denote $\mathbf{a}^{'}$ by the predecessor of $\mathbf{a}$. Then $u_{\mathbf{a}}$ determines a linear map $i_{\mathbf{a}}: M^{\mathbf{a}}\longrightarrow M^{0}=Wh_{\eta}(M)$ by sending $m\in M^{\mathbf{a}}$ to $u_{\mathbf{a}}m$. It is clear that $i_{\mathbf{a}}$ is surjective and $Keri_{\mathbf{a}}=M^{\mathbf{a}^{'}}$, when $\bf a \neq 0$. Note that each $M^{\bf a}$ is stable under $\rm{End}_{U(\ggg)}$-module, hence we get the folowing isomorphism of $\rm{End}_{U(\ggg)}$-module $M^{\mathbf{a}}/M^{\mathbf{a}^{'}}\simeq M^{0}$.
			
			Now we apply the above observation to the $\mf g$-module  $M$, where $M=Q(\eta) \cong U(\mf g)/I_\eta.$ Since we have $Q(\eta)$ is a free $U(\mf g, \eta)$-module of rank 1, so is $Q(\eta)^{\bf a}/ Q(\eta)^{\bf a'}$ for all $\bf a \neq 0$. In particular, the image of $X^{\bf a}$ in $Q(\eta)^{\bf a}/ Q(\eta)^{\bf a'}$ is a free generator of $Q(\eta)^{\bf a}/ Q(\eta)^{\bf a'}$. Now since $Q(\eta)= \cup_{{\bf a} \in \N^{n-1}} Q(\eta)^{\bf a}$, it follows that $X^{\mathbf{a}}1_{\eta}$ forms a basis of $Q(\eta)$ over $U(\ggg,\eta)$.
			
			Define a $U(\ggg)$-module map
			$$\mu:Q(\eta)\otimes_{U(\ggg,\eta)}V\longrightarrow M$$
			$$u1_{\eta}\otimes v\longmapsto uv.$$
			Since every element of $Q(\eta)\otimes_{U(\ggg,\eta)}V$ can be uniquely expressed as $\sum X^{\mathbf{a}}1_{\eta}\otimes v_{\mathbf{a}}$, it is easy to check that $\mu$ is a $U(\ggg)$-module map. Further, it clear that $\mu$ is surjective and injectivity follows by using the map $\phi$.
			
			Now we need to show that for any $M=Q(\eta)\otimes_{U(\ggg,\eta)}V^{'}\in\mathscr{W}^{'}_{\eta}$, $V^{'}\simeq Wh_{\eta}(M)$ as $U(\ggg,\eta)$-module. Define
			$$\nu: V^{'}\longrightarrow Wh_{\eta}(M)$$
			$$v\longmapsto 1_{\eta}\otimes v$$
			Since any element $m\in M$ can be expressed uniquely as $\sum X^{\mathbf{a}}\nu(v_{\mathbf{a}})$ and $m\in Wh_{\eta}(M)$ only if $\nu(v_{\mathbf{a}})=0$ for all $\mathbf{a}\ne 0$, $\nu$ is a $U(\ggg,\eta)$-module isomorphism. This completes the proof of Case I. \\
			{\bf Case II:} Let $\eta$ be any non zero character. Then $\eta \neq 0$ for some simple root vectors between $1 \leq k \leq n-1$. Let $\eta \neq 0$ on the root vectors corresponding to the $k$ values $i_1, \cdots, i_r$. Then we consider corresponding $Y_{i_j}$ and $X_{i_j}$ for $1 \leq j \leq r$ as Case I, which will satisfy the property $\eta([Y_{i_j},X_{i_k}])=\de_{jk}$ and then continue the proof of Case I. This completes the proof.
		\end{proof}
		
		Note that $\mscr W_{\eta}$ is the subcategory of $\mscr W_{\eta}^{'}$ and any irreducible object $\mcal L(\eta,\la)$ in $\mscr W_{\eta}$ is also an irreducible object in $\mscr W_{\eta}'$. Now by the above equivalence, we get a collection of irreducible Whittaker modules in $U(\ggg,\eta)\text{-mod}$.
		
		\begin{cor}
			For any irreducible object $\mcal L(\eta,\la)$ in the category $\mscr W_{\eta}$, the Whittaker space $Wh_{\eta}(\mcal L(\eta,\la))$ is also an irreducible object in the category $U(\ggg,\eta)\text{-mod}$.
		\end{cor}
		
 Very recently,  Futorny and Tantubay studied Whittaker modules for infinite-dimensional Lie super algebras $W(m;n)$.  Their consideration is quite different from ours.


\begin{thebibliography}{99}
			\bibitem{Back} E. Backelin, {\em Representation of the category O in Whittaker categories}, Int. Math. Res. Notices 4  (1997),
			153–172.

	\bibitem{CWC} C.W. Chen, {\em Whittaker Modules for Classical Lie Superalgebras}, Commun. Math. Phys.  388 (2021), 351–383.
			\bibitem{CW12}  S.-J. Cheng and W. Wang, {\it Dualities and representations of Lie \textit{superalgebras}} Grad. Stud. Math., 144, Amer.
			Math. Soc., Providence, RI, 2012.

			\bibitem{DSY} F.-F. Duan, B. Shu and Y.-F. Yao, {\em Parabolic BGG categories and their block decomposition for
				Lie superalgebras of Cartan type}, J. Math. Soc. Japan 76 (2024), no. 2, 503-562.



		\bibitem{FS} V. Futorny and S. Tantubay, Whittaker Modules For W Type Cartan Lie Superalgebras,  arXiv:2511.17995.

			\bibitem{Kac77} V. Kac, {\itshape Lie superalgebras}, Adv. Math., 26 (1977), 8-96.
			\bibitem{K}	 B. Kostant; {\em On Whittaker vectors and representation theory}, Invent. Math. 48 (1978); 101–184.
			\bibitem{Mc} E. McDowell, On Modules Induced from
			Whittaker Modules, J. Algebra, 96 (1985), 161-177.

			
			\bibitem{MS} D. Milicic and W. Soergel, {\em  The composition series of modules induced from
				Whittaker modules}, Comment. Math. Helv. 72 (1997) 503-520.
			
			\bibitem{P} A. Premet, {\em Special transverse slices and their enveloping algebras}, Adv. Math. 170 (2002),
			1–55.
			
		
			
			\bibitem[10]{SZ} Y. Su and R.B. Zhang, Generalised Jantzen filtration of Lie super algebras,  J.Eur.Math.Soc. 14 (2012),1103–1133.
			
	
			\bibitem{ZS} Y. Zeng and B. Shu, {\em Finite W-superalgebras for basic Lie superalgebras}, J. Algebra 438
			(2015), 188–234.
		\end{thebibliography}
	\end{document}